\documentclass[12pt,psamsfonts]{amsart}
\usepackage{amsmath,amsthm,amsfonts,amssymb}
\usepackage{eucal}
\usepackage{graphicx}

\numberwithin{equation}{section}

\newtheorem{theorem}[equation]{Theorem}

\newtheorem{corollary}[equation]{Corollary}
\newtheorem{lemma}[equation]{Lemma}

\newtheorem{prop}[equation]{Proposition}
\theoremstyle{definition}

\newtheorem{rmk}[equation]{Remark}

\newtheorem{example}[equation]{Example}
\newtheorem{defn}[equation]{Definition}

\newcommand{\g}{\mathfrak{g}}

\newcommand{\R}{{\mathbb{R}}}
\newcommand{\C}{{\mathbb{C}}}
\newcommand{\Q}{{\mathbb{Q}}}

\newcommand{\Z}{{\mathbb{Z}}}

\newcommand{\im}{\mathrm{i}}

\newcommand{\ot}{\otimes}
\newcommand{\Rep}{\mathrm{Rep}}
\newcommand{\Spec}{\mathrm{Spec}}
\newcommand{\GL}{\mathrm{GL}}

\newcommand{\End}{\mathrm{End}}
\newcommand{\Hom}{\mathrm{Hom}}

\newcommand{\SL}{\mathrm{SL}}
\newcommand{\PSL}{\mathrm{PSL}}
\newcommand{\Sp}{\mathrm{Sp}}
\newcommand{\PSp}{\mathrm{PSp}}

\newcommand{\1}{{1 \hspace{-0.35em} {\rm 1}}}

\newcommand{\mS}{\mathcal{S}}

\newcommand{\B}{\mathcal{B}}
\newcommand{\CC}{\mathcal{C}}

\newcommand{\lan}{\langle}
\newcommand{\ra}{\rangle}
\newcommand{\ve}{\varepsilon}

\newcommand{\la}{{\lambda}}

\begin{document}
\title[Finite Quotients of $\B_3$]
{Finite Linear Quotients of $\B_3$ of Low Dimension}

\author{Eric C. Rowell}
\email{rowell@math.tamu.edu}
\address{Department of Mathematics\\
    Texas A\&M University \\
    College Station, TX 77843-3368\\
    U.S.A.}
\author{Imre Tuba}
\email{ituba@mail.sdsu.edu}
\address{San Diego State University\\
Imperial Valley Campus\\
720 Heber Ave.\\
Calexico, CA 92231\\
U.S.A.}
\thanks{Eric Rowell was partially supported by NSA grant H98230-08-1-0020.}

\begin{abstract}
We study the problem of deciding whether or not the image of an irreducible representation of the braid group $\B_3$ of degree
$\leq 5$ has finite image if we are only given the eigenvalues of a generator.   We provide a partial algorithm that determines when the images are finite or
infinite in all but finitely many cases, and use these results to study examples coming from quantum groups.  Our technique uses two classification theorems and the computational group theory package GAP.
\end{abstract}
\maketitle
\section{Introduction}
Let $\B_3$ denote Artin's braid group on $3$ strands with generators $\sigma_1$ $\sigma_2$, satisfying
 $$\sigma_1\sigma_{2}\sigma_1=\sigma_{2}\sigma_1\sigma_{2}.$$

We consider the following problem.
Suppose $\phi:\B_3\rightarrow GL(V)$ is a $d$-dimensional complex representation  for which we are given:
\begin{enumerate}
\item $\rho$ is irreducible and
\item $\Spec(\rho(\sigma_1))=\{\la_1,\ldots,\la_d\}.$
\end{enumerate}
Can we determine if $G:=\rho(\B_3)$ is a finite or infinite group from this
information?  In this paper we
determine some conditions under which $G$ is finite or infinite, assuming that
$d\le 5$.

The general question of determining the image of complex braid group representations
seems first to have been studied by Jones \cite{Jones1} for the representations
associated with the Jones polynomial.  Indeed, the Burau representation was essentially the only representation of $\B_n$ that was known until the 1980s.  More recently this question has been studied extensively for unitary representations
obtained from solutions to the Yang-Baxter equation \cite{FRW}, Hecke algebras
\cite{FLW1,FLW2}, and BMW-algebras \cite{LRW,LR}.  In all of these cases one has a tower  of (generally reducible) representations $\rho_n$ such that $$\rho_3(\B_3)\subset\cdots\subset\rho_n(\B_n)\subset\rho_{n+1}(\B_{n+1})\subset\cdots.$$  So clearly it is enough to show that $\rho_3(B_3)$ is infinite to conclude the same for the tower.  One motivating application of this is to answer the question of universality in the setting of topological quantum computation (see \cite{FLW2}).

A generalization of the above towers of representations comes from \emph{ribbon categories}, see \cite{TurBook} for the definition.  For any object $X$ in a ribbon category $\CC$ one obtains
a tower of representations of $\B_n$ acting on $\End(X^{\ot n})$.  If the object is a self-dual object (i.e. $X \cong X^*$) then $\B_3$ acts on $V=\Hom(X,X^{\ot  3})$, often irreducibly.  As an application, we will apply our results to ribbon categories coming from quantum groups in Section \ref{aps}.
\section{Main Result}
Throughout this paper let $\rho:\B_3\rightarrow\GL(V)$ be a $d$-dimensional irreducible representation with $2\leq d\leq 5$ and set $A=\rho(\sigma_1)$ and $B=\rho(\sigma_2)$.  Let $G$ denote the image  $\rho(\B_3)=\lan A,B\ra$ i.e. the group generated by $A$ and $B$ and
define $\mS:=\Spec(A)=\Spec(B)$.  The analysis naturally breaks into imprimitive and primitive cases (defined in Section \ref{genres}).
Our results are summarized in the following:
\begin{theorem}\label{maintheorem}  Let $\rho$, $G$, $A$ and $\mS$ be defined as
in the previous paragraph.
Let $\mS=\{\la_1,\ldots,\la_d\}$, and define the projective order of $A$ by
$$po(A):=\min\{t: (\la_1)^t=(\la_2)^t=\cdots=(\la_d)^t\}.$$   We use the convention
that each successive statement excludes the hypotheses of all
of the preceding cases.
\begin{enumerate}
\item[(a)] Suppose some $\la_i$ is not a root of unity, or $\la_i=\la_j$ for some $i\neq j$.  Then $G$ is \textbf{infinite}.
\item[(b)] Suppose $po(A)\leq 5$.  Then $G$ is \textbf{finite}.
\item[(c)] Suppose $G$ is imprimitive.  Then $\mS$ is of the form:
\begin{enumerate}
 \item[(i)] $\{\pm\chi,\alpha\}$ or $\chi\{1,\omega,\omega^2\}\cup\{\alpha\}$
with $\omega$ a primitive 3rd root of unity and $G$ is finite or
\item[(ii)] $\{\pm r,\pm s\}$.  In this case if $u=r/s$
is a root of unity of order $o(u)\in \{7,8,9\}\cup[11,\infty)$ then $G$ is \textbf{infinite}, if $o(u)=6$, $G$ is \textbf{finite} and if $o(u)=5$ or $10$
one cannot decide $|G|$ without further information.
\end{enumerate}
\item[(d)] Suppose $G$ is primitive. Then:
\begin{enumerate}
\item[(i)] If $d=2$ then $G$ is \textbf{infinite}.
\item[(ii)] If $d=3$ and $po(A)\geq 8$ then $G$ is \textbf{infinite}.  If $po(A)=7$, and $\frac{1}{\la_1}\mS$ is Galois conjugate to $\{1,e^{2\pi\im/7},e^{2k\pi\im/7}\}$ with $k$ even, $G$ is \textbf{infinite}, whereas if
$k$ is odd, $G$ is \textbf{finite}.
\item[(iii)] If $d=4$ and $po(A)\not\in\{6,\ldots,10,12,15,20,24\}$ then $G$ is \textbf{infinite}.
\item[(iv)] If $d=5$ and $po(A)\in\{7,8\}\cup[13,\infty)$ then $G$ is \textbf{infinite}.
\end{enumerate}
\end{enumerate}
\end{theorem}

\begin{rmk}
Notice that this theorem covers all but the following cases: $G$ primitive with
\begin{enumerate}
\item $d=4$ with $po(A)\in\{6,\ldots,10,12,15,20,24\}$.
\item $d=5$ with $po(A)\in\{6,9,10,11,12\}$.
\end{enumerate}
So Theorem \ref{maintheorem} can be used as an algorithm to decide the if $G$ is finite or not, at least for $d\le 3$ or $G$ imprimitive.
\end{rmk}

There are two main ingredients to our approach.  Necessary
conditions can be derived from the classification of finite primitive linear groups of low degree (see \cite{Ft}), while
sufficient conditions can often be gleaned from the following classification of irreducible representations of $\B_3$ for $2\leq d\leq 5$ found in \cite{TbWz}:
\begin{theorem}[Tuba-Wenzl]\label{tw}

 \begin{enumerate}
  \item Suppose $\rho$ is an irreducible representation of $\B_3$ with eigenvalues
$\mS=\{\la_1,\ldots,\la_d\}$
as above.  Then,
\begin{enumerate}
\item $\rho$ is uniquely determined up to equivalence by $\mS$ up to a choice of $\gamma^2$ for $d=4$ and a choice of $\gamma$ for $d=5$ where $\gamma:=\det(A)^{1/d}$.
\item there exists a basis for $\C^d$ so that the matrices $A$ and $B$ are in a triangular form, given in
\cite{TbWz}.
\end{enumerate}
\item There exists an irreducible representation of $\B_3$ with eigenvalues $\mS$ if only if the $\la_i$ and $\gamma$ do not satisfy certain polynomials.  In particular for $d=3$ and $d=4$ these polynomials are:
\begin{enumerate}
\item $d=3$ $\la_r^2+\la_s\la_t$ for $r,s,t$ distinct and
\item $d=4$ $(\la_r^2+\gamma^2)(\gamma^2+\la_r\la_s+\la_t\la_u)$ for $r,s,t,u$ distinct.
\end{enumerate}
\end{enumerate}
\end{theorem}
We will refer to this theorem as the \emph{TW classification}.

\section{General Results}\label{genres}
In this section we state some general results that will be used later.

\begin{lemma}Suppose $|G|<\infty$.  Then, $\mS$ consists of $d$ \emph{distinct} roots of unity.
\end{lemma}
\begin{proof}
 It is clear that if $A$ is of finite order, it is diagonalizable and its eigenvalues must be roots of unity.  Corollary 2.2 of  \cite{TbWz} states that the minimal and characteristic polynomials of $A$ coincide (for $d\leq 5$) so that the eigenvalues of $A$ must also be distinct.
\end{proof}

The following result due to Coxeter \cite{Cox} dates back to the 1950s:
\begin{prop}\label{coxprop}
The quotient of $\B_n$ by the normal subgroup generated by
$\sigma_1^p$ is finite if and only if $1/p+1/n>1/2$.
In particular, the quotient of $\B_3$ by the normal subgroup generated by $\sigma_1^p$
is finite if and only if
$p=2,3,4$ or $5$, where the quotient groups are
\begin{eqnarray}\label{qgps}
\B_3/\lan \sigma_1^p\ra\cong\begin{cases}S_3 & p=2\\
\SL(2,3) & p=3\\
H & p=4\\
\SL(2,5)\times\Z_5 & p=5
                             \end{cases}\end{eqnarray}
where $H$ is a non-split central extension of $S_4$ by $\Z_4$.
\end{prop}

We have the following immediate consequence:
\begin{corollary}\label{coxcor}
 If $\mS\subset\chi\{\zeta_p^i:i=1...p\}$ where $\chi$ is any root of unity and $\zeta_p$ is a primitive
$p$th root of unity with $2\leq p\leq 5$ then $|G|$ is finite.
\end{corollary}
\begin{proof}
 The given hypotheses imply that $G$ is contained in a central extension of a quotient of one of the finite groups in (\ref{qgps}) by the finite cyclic group generated by $\chi$.
\end{proof}

Irreducible finite linear groups naturally break up into two distinct classes: primitive and imprimitive groups.

\begin{defn}
A linear group $\Gamma\subset\GL(V)$ is \textbf{imprimitive} if $V$ is irreducible
and can be expressed as a direct sum of
subspaces $V_i$ which $\Gamma$ permutes nontrivially.  Otherwise, we say that $\Gamma$ is \textbf{primitive}.
\end{defn}

The following observation can be found in \cite{LRW}:
\begin{lemma}\label{nocyclem}
 Suppose $G$ is an irreducible imprimitive finite linear group of dimension $d$.  Then $\mS$ contains a $\C$-coset of the group $C_r=\{\zeta_r^i: 0\leq i\leq r-1\}$ of roots of unity of degree
$r\leq d$.
\end{lemma}

The imprimitive linear groups that appear as images of irreducible representations of $\B_3$ will be analyzed in Section \ref{imp}, and the primitive cases will be covered in Section \ref{prim}.

\section{Imprimitive Groups}\label{imp}

For $2\leq d\leq 5$, there are only two ways that a group $\Gamma\subset \GL(V)$ can be imprimitive:
\begin{enumerate}
\item[Case (1)] $d=4$ and $V=V_1\oplus V_2$ with $\dim(V_i)=2$ and $\Gamma$ permutes $V_1$ and $V_2$ non-trivially and this block structure has no
refinement to $1\times 1$ blocks or
\item[Case (2)] $\Gamma$ is a isomorphic to a monomial group, i.e. a subgroup of $M(d):=S_d\ltimes D(d)$ where $D(d)$ is the group of $d\times d$ diagonal matrices and $S_d$ is identified with the $d\times d$ permutation matrices, and acts of $D(d)$ by permuting the entries.
\end{enumerate}

Case (2) is covered by the following:
\begin{theorem}\label{monom}  Let $2\leq d\leq 5$, and assume $\mS$ consists of roots of unity.  If $G$ is an irreducible imprimitive monomial group then $G$ is finite and $\mS$ is of the form:
\begin{enumerate}
 \item[(a)] $\mS=\chi\{\zeta_d^i: 1\leq i\leq d\}$,
\item[(b)] $\mS=\{\alpha,\pm\chi\}$ or
\item[(c)] $\mS=\{\alpha\}\cup\chi\{1,\zeta_3,\zeta_3^2\}$.
\end{enumerate}
\end{theorem}
\begin{proof}
In each of the cases (a)-(c) one checks that the image of $G$ is indeed finite,
in case (a) it follows from Proposition \ref{coxcor}, while in cases (b) and (c) one employs the TW classification (in case (c) one has two cases to check depending on the sign of $\det(A)^{1/2}$).  By Lemma \ref{nocyclem} we need only
show that if $\mS$ contains a coset of roots of unity but is not in case (a)-(c) the corresponding representation is not irreducible.  Since $G$ is monomial, we may assume that its generators $A$ and $B$ are generalized permutation matrices, that is, of the form $A=D_1P_1$, $B=D_2P_2$ where $D_i$ is a diagonal matrix and $P_i$ is a permutation matrix.  It is clear that $G$ is reducible if the action of $g:=\lan P_1, P_2\ra$ on the standard basis for $\C^d$ is intransitive, since the span of any orbit will be an invariant subspace for $G$.  For the excluded
cases we may assume either $d=4$ and $P_1$ is $(1,2)$ or $(1,2)(3,4)$,
or $d=5$ and $P_1$ is $(1,2)$, $(1,2,3)$, $(1,2)(3,4)$, $(1,2,3)(4,5)$ or $(1,2,3,4)$.  One then uses the braid relation to determine
the possible cycle forms of $P_2$.  We record them in Table \ref{mapletable} in disjoint cycle notation.
In each case it is clear that $P_1$ and $P_2$ generate an intransitive subgroup
of $S_d$.
\end{proof}

\begin{table}\caption{Intransitive Cases}\label{mapletable}
\begin{tabular}{*{3}{|l}|}
\hline
$d$ & $P_1$ & $P_2$  \\
\hline\hline
4& $(1,2)$ & $(1,i)$ or $(2,j)$\\
\hline
4 & $(1,2)(3,4)$ & $P_1$\\
\hline
5 & $(1,2)$ & $(1,i)$ or $(2,j)$\\
\hline
5 & $(1,2)(3,4)$ & $(1,2)(j,k)$ or $(3,4)(r,s)$\\
\hline
5 & $(1,2,3)$ & $P_1$, $(1,j,2)$, $(2,j,3)$ or $(1,3,j)$ $j=4$ or $5$\\
\hline
5 & $(1,2,3)(4,5)$ & $P_1$ \\
\hline
5& $(1,2,3,4)$ & $(i,j,k,\ell)\in S_4$\\
\hline
\end{tabular}
\end{table}

Now consider Case (1), so that $A$ and $B$ permute two dimension $2$ vector spaces $V_1$ and $V_2$.  Since $G$ is assumed to be irreducible, by choosing an ordered basis consisting of the union of bases of $V_1$ and $V_2$, we may assume either $A$ or $B$ is block skew-diagonal with blocks of size $2$.  Otherwise both $A$ and $B$ would be block diagonal with respect to this basis, violating irreducibility.  Now the characteristic polynomial of such a block skew-diagonal matrix is a polynomial in $x^2$, so that eigenvalues occur in pairs $\pm r$ and $\pm s$.  If $V_1\oplus V_2$ has a refinement to $1\times 1$ blocks then $\mS=\chi\{\pm 1,\pm i\}$ and is covered by Theorem \ref{monom}.
 Thus we may assume $\mS=\{\pm r,\pm s\}$ with $r/s$ not a $4$th root of unity.    For each pair $(r,s)$ there are
two inequivalent irreducible $4$-dimensional representations of $\B_3$ with $\mS=\{\pm r,\pm s\}$ depending on a choice of $D=\pm\sqrt{\la_2\la_3/(\la_1\la_4)}=\pm 1$ (see \cite[Prop. 2.6]{TbWz}).  Setting $$p_1(u,t):=u^2t^4+(u+u^2+u^3)t^3+(1+2u+2u^2+2u^3+u^4)t^2+(u+u^2+u^3)t+u^2$$ we have the following:
\begin{theorem}
Suppose $G$ is imprimitive and irreducible with unrefinable blocks of size 2 so that $\mS=\{\pm r,\pm s\}$ with $u=r/s$ not a $4$th root of unity.  Denote the order of $u$ by $o(u)$.  Then:
\begin{enumerate}
\item[(a)] The characteristic polynomial of $AB^{-1}$ is either $p_1(u,t)/u^2$
 or
$p_1(-u,t)/u^2$ where $D=-1$ corresponds to $p_1(u,t)/u^2$ and $D=1$ corresponds
to $p_1(-u,t)/u^2$.
\item[(b)] If $o(u)\not\in\{3,5,6,10\}$, then $G$ is \textbf{infinite}.
\item[(c)] If $(D,o(u))\in\{(1,5),(-1,10)\}$, $G$ is \textbf{infinite}.
\item[(d)] If $o(u)\in\{3,6\}$ $G$ is \textbf{finite}.
\item[(e)] If $(D,o(u))\in\{(-1,5),(1,10)\}$ $G$ is \textbf{finite}.

\end{enumerate}
\end{theorem}

\begin{proof} We may replace $A$ by $\frac{1}{s}A$ and $B$ by $\frac{1}{s}B$
without changing the finiteness of $G$ or the characteristic polynomial of $AB^{-1}$, and doing so we obtain the following matrices from \cite{TbWz}:
$$A=\begin{pmatrix}1&-({D}^{-2}+D^{-1}+1)&
\left( {D}^{-2}+D^{-1}+1 \right) u&-u\\0&-1&
\left( D^{-1}+1 \right) u&-u\\0&0&u&-u
\\0&0&0&-u\end {pmatrix}$$
and
$$B=\begin{pmatrix} -u&0&0&0\\-u&u&0&0
\\-D&D+1&-1&0\\-{D}^{3}&{D}^{3}+{D
}^{2}+D&-{D}^{2}-D-1&1\end {pmatrix}$$
and (a) follows by computation.
If the eigenvalues of $AB^{-1}$ are not roots of unity then $G$ is infinite.
Let us assume that $D=-1$ and the roots $t$ of $p_1$ (i.e. the eigenvalues of $AB^{-1}$) are roots of unity.  Set $x=t+1/t$ and $y=u+1/u$ and consider
$p_1(u,t)/(ut)^2=x^2+(y+1)x+(y^2-2+2y)$.  By applying a Galois automorphism of the field $\Q[u,t]$ we may assume that $u=e^{2\pi \im/\ell}$ for $\ell=o(u)$. Notice that $y=2\Re(u)=2\cos(2\pi/\ell)$.  If we assume $\ell>6$
then $y>1$, and the discriminant of $x^2+(y+1)x+(y^2-2+2y)$ is:
$-3y^2-6y+9<0$ so that $x\not\in\R$, contradicting $x=2\Re(t)$.
Next assume $D=1$.  The argument is essentially the same, except we get
the polynomial $x^2+(1-y)x+(y^2-2-2y)$ with discriminant $9-3y^2+6y$ when we replace $u$ by $-u$ in $p_1$.
Using a Galois automorphism we may assume that $u=e^{2\pi \im k/\ell}$ where $k$ is chosen so that $u$ is nearest $-1$.  Now $y=2\Re(u)<-1$ provided $\ell>4$ and $\ell=\neq 6$ or $10$, the latter exclusion coming from the fact that the primitive $10$th root of unity nearest $-1$ is $e^{6\pi \im/10}$ which has real part $-\cos(2\pi/5)>-1/2$.  Again, one obtains a contradiction since $9-3y^2+6y<0$
contradicting $x\in\R$.  This proves (b) and (c).

To prove (d) we observe that in case $o(u)=3$ or $6$, the projective order of $A$ is $6$, and the element $AB^{-1}$ also has projective order $6$.  These together with the braid relation $ABA=BAB$ are enough to conclude that, modulo the center, $G$ is a quotient of
a group of order $648$.

For (e)  we consider the normal subgroup $H$ generated by $A^5$ and $B^5$.  Observe that $H$ is of finite index by the result of Coxeter above.  A set of generators for $H$ is $\{A^5,B^5,AB^5A^{-1},BA^5B^{-1}\}$.  Set $M:=\begin{pmatrix} 1&-1\\0&-1\end{pmatrix}$, $N:=\begin{pmatrix} -1&0\\-1&1\end{pmatrix}$ and
$T:=\begin{pmatrix} 0&1\\1&0\end{pmatrix}$.  Then if $o(u)=5$ and $D=-1$ we compute $A^5=M\oplus M$, $B^5=N\oplus N$ and $AB^5A^{-1}=BA^5B^{-1}=T\oplus T$.  One can easily see that $M,N$ and $T$ generate
$S_3$, from which it follows that $G$ is finite.  For the case $D=1$ and $o(u)=10$ the matrices involved are slightly more complicated, such as $\begin{pmatrix} 1 & -3\\0&-1\end{pmatrix}$, $\begin{pmatrix} -2 & 3\\-1&2\end{pmatrix}$ and $\begin{pmatrix} 1 & 0\\1&-1\end{pmatrix}$, but give an equivalent representation of $S_3$ and we conclude that $G$ is finite in this case as well.

\end{proof}

This completes the proof of Theorem \ref{maintheorem}(c).

 \section{Primitive Groups}\label{prim}
 In this section we
assume that $G$ is a finite primitive irreducible group.  Then by rescaling $A$ and $B$ by a choice of a root of unity $(\det(A))^{-1/d}$, we may assume that
$G$ is unimodular without changing $po(A)$.   Thus we can determine the possible values of $po(A)$ by computing the projective orders of elements in the groups on Feit's list of finite unimodular primitive irreducible linear groups of degree $5$ or less \cite{Ft}.

\begin{lemma}\label{orders}
Suppose $g$ is an element in a primitive unimodular irreducible finite group $H$ of dimension $2\leq d\leq 5$ and let $t$ be the order of $g$ modulo $Z(H)$.  Then
\begin{enumerate}
\item[(a)] if $d=2$, $1\leq t\leq 5$,
\item[(b)] if $d=3$, $1\leq t\leq 7$,
\item[(c)] if $d=4$, $t\in\{1,\ldots,10,12,15,20,24\}$ and
\item[(d)] if $d=5$, $t\in\{1,\ldots,6,9,\ldots,12\}$.
\end{enumerate}
\end{lemma}
\begin{proof}
The list of such  primitive unimodular irreducible finite groups is given in \cite[Section 8.5]{Ft}, and the only work to do is to construct the groups using \cite{GAP} and compute the orders of elements in the quotient $H/Z(H)$.  Modulo their centers, the dimension $2$ groups are  $S_4, A_4$, and $A_5$, and for $d=3$ one has $A_5, A_6, \PSL(2,7)$ and subgroups of the Hessian group of order $216$.  For $d=4$ there are 11 types of groups.  Four of these come from direct products of dimension $2$ linear groups and their extensions by an order $2$ outer automorphism.  One also has the explicit groups $A_5, S_5, A_6, S_6, A_7, \SL(2,5), \SL(2,7)$ and $\Sp(4,3)$.  The last class of dimension $4$ groups are certain subgroups of extraspecial $2$-groups of order $2^5$ by their automorphism groups.
For $d=5$ Feit's list yields the following groups (with trivial centers):
\begin{enumerate}
\item Subgroups of $\SL(2,5)\ltimes(\Z_5\times\Z_5)$: orders $(1,2,3,4,5,6,10)$
\item $A_5$: orders $(1,2,3,5)$
\item $S_5$: orders $(1,2,3,4,5,6)$
\item $A_6$: orders $(1,2,3,4,5)$
\item $S_6$: orders $(1,2,3,4,5,6)$
\item $\PSp(4,3)$: orders $(1,2,3,4,5,6,9,12)$
\item $\PSL(2,11)$: orders $(1,2,3,5,6,11)$
\end{enumerate}

To illustrate how one uses GAP to get the information listed in the theorem, we give the following sample code for determining the possible orders in
the class of $4$ dimensional groups mentioned above.:\\

\noindent\texttt{gap> H:=ExtraspecialGroup(32,"-");\\
<pc group of size 32 with 5 generators>\\
gap> T:=AutomorphismGroup(H);\\
 <group of size 1920 with 2 generators>\\
 gap> M:=SemidirectProduct(T,H);\\
 <permutation group with 7 generators>\\
gap> N:=M/Center(M);\\
<permutation group of size 30720 with 11 generators>\\
gap> s:=Elements(N);;\\
gap> ords:=List(s,Order);;\\
gap> Set(ords);\\
$[$ 1, 2, 3, 4, 5, 6, 8, 12 $]$ }

\end{proof}

Parts (i),(iii) and (iv) of Theorem \ref{maintheorem}(d) are immediate from Lemma \ref{orders}.  We proceed by cases to prove part (ii).
\begin{proof}(of Theorem 2.1(d)(ii)).
For $d=3$, Lemma \ref{orders} implies that $G$ is infinite if $po(A)\geq 8$, so we need only consider $\ell:=po(A)=6$ and $7$.
Let $\mS=\{\la_1,\la_2,\la_3\}$.  Clearly $G$ is finite and irreducible if and only if the group generated by $A/\la_1$ and $B/\la_1$ is finite and irreducible, so we may
assume $\mS=\{1,\theta,\phi\}$.  By applying a Galois automorphism we may further assume that $\theta=e^{2\pi\im/\ell}$ and $\phi=\theta^k$ with $2\leq k\leq \ell-1$.
From \cite[Proposition 2.5]{TbWz} we obtain:
$$A=\begin{pmatrix} 1&\phi/\theta+\theta & \theta\\
            0&\theta & \theta\\
            0&0& \phi\end{pmatrix},
B=\begin{pmatrix} \phi & 0& 0\\
        -\theta & \theta & 0\\
        \theta & -\phi/\theta-\theta & 1
  \end{pmatrix}.$$

If $po(A)=6$, the cases $k=2,5$ correspond to reducible representations (with infinite image), since $\theta\cdot 1+\phi^2=e^{\pi\im/3}+e^{2k\pi\im/3}=0$ for $k=2,5$.  For $k=3,4$ the eigenvalues case are $\pm 1,e^{\pi\im/3}$ and $1,\pm e^{\pi\im/3}$ respectively.  Theorem \ref{monom} suggests these give us imprimitive groups.  Writing down the $A$ and $B$ matrices in monomial form one checks easily that they have no common eigenvector and hence the representation is irreducible and imprimitive. Thus $po(A)=6$ implies $G$ is either reducible or imprimitive.

If $po(A)=7$, the cases $k=2,4$ or $6$ each give infinite groups as can be seen either by checking that the eigenvalues of $AB^{-1}$ are not roots of unity, or by observing that that the only dimension $3$ irreducible imprimitive group with elements of order 7 is $\PSL(2,7)$, which also has elements of orders $1,2,3$ and $4$.  So we need only check that $AB^{-1}$ does not have one of these orders, and for $k=2,4$ or $6$ this computation yields the desired result. For $k=3,5$ one finds that $(AB^{-1})^4=I$ which, together with the braid relation and $A^7=B^7=I$, is enough to show that $G$ is a quotient of a group of order $1176$ and hence finite.
\end{proof}

This completes the proof of Theorem \ref{maintheorem}(d).

\begin{rmk}
We would like to point out the limitations of our approach, in particular why we do not get complete results for dimensions $4$ and $5$ with $G$ primitive.  Firstly, the representations are not uniquely determined by the eigenvalues: for $d=4$ there are two choices, and for $d=5$ there are $5$.  Secondly, the sets of eigenvalues $\mS$ with $po(A)=t$ for some $t$ found in Lemma \ref{orders} can be quite large particularly if $t$ is composite.  Finally, the only technique we have for showing that $G$ is infinite in these cases is to show that some element (such as $AB^{-1}$) has infinite order.  Moreover, if the order happens to be finite, this relation might not be sufficient to conclude that
$G$ is finite, in which case we must resort to further \emph{ad hoc} means.  For $d\geq 6$, there are further issues.  A representation of degree $6$ or higher is not determined by $\mS$ and $\gamma$.  Moreover, the imprimitive cases are significantly more delicate.  Given the eigenvalues of an irreducible representation $\rho$ of $\B_3$ of degree $d=6$ or $7$ one can sometimes use our approach as follows: first verify that $\rho(\B_3)$ is primitive by using Lemma \ref{nocyclem}.  Then use Feit's list \cite{Ft} to determine which $po(A)$ can appear in finite linear groups of degree $d$.  For example, if $d=7$, $\mS$ does not contain a full coset of $C_r\subset\C$, $r\leq 7$ and $p\mid po(A)$ for $13\neq p\geq 11$ a prime, then $\rho(\B_3)$ is infinite.
\end{rmk}

\section{Applications}\label{aps}
We apply our results to ribbon categories obtained as subquotients of $\Rep(U_q\g)$ where $q=e^{\pi\im/\ell}$, see \cite{TurBook} for details.  These categories
will be denoted $\CC(\g,q,\ell)$ as in \cite{Rsurvey}.  In particular the representation of $\B_3$ acting on $\Hom(V,V^{\ot 3})$ is irreducible provided the eigenvalues of $A$ are distinct by \cite[Lemma 3.2]{TbWz}.  We can compute the
eigenvalues by applying results of Reshetikhin found in \cite[Corollary 2.22(3)]{LeducRam}.  Since the cases $d=2$ and $d=3$ were already considered in \cite{FLW2} and \cite{LRW} respectively, we focus on the cases $d=4$ and $5$.
Since the order of $G=\rho(\B_3)$ is invariant under Galois automorphisms, we can safely assume $q=e^{\pi\im/\ell}$ without loss of generality.  The computation involves two steps: first we use standard Lie theory techniques to decompose $V^{\ot 2}$ as a direct sum of simple objects $V_{\la_i}$.  The eigenvalues of the
action of $\rho(\sigma_1)$ are then computed (up to an overall scale factor)
as $\pm q^{\lan \la_i+2\delta,\la_i\ra/2}$ where the sign is positive if $V_{\la_i}$ appears in the symmetrization of $V^{\ot 2}$ and is negative otherwise, and $\delta$ is $1/2$ the sum of the positive weights.

\begin{example} Let $\g=\g_2$ and let $V$ be the simple object labelled by $\la_1=(\ve_1-\ve_3)$, i.e. the highest weight of the $7$ dimensional fundamental representation of $\g_2$.  First suppose that $3\mid\ell$ so that the corresponding category is a unitary modular category (see \cite{rowellJPAA}).  Then if $18\leq\ell$, $\dim\Hom(V,V^{\ot 3})=4$ and the image of $\sigma_1$ has eigenvalues $\{q^{-12},q^2,-q^{-6},-1\}$.  Thus the projective order of the image of $\sigma_1$ is $\ell$ if $\ell$ is even and
$2\ell$ if $\ell$ is odd.  Thus by Theorem \ref{maintheorem}(d)(iii) the image of $\B_3$ is infinite unless or $24$.  Notice however, that for $\ell=24$ we have repeated eigenvalues, which implies the image is infinite by Theorem \ref{maintheorem}(a).  Now suppose that $3\nmid\ell$.  In this case
$\dim\Hom(V,V^{\ot 3})=4$ for $10\leq\ell$.  The eigenvalues of the image of $\sigma_1$ are as above, so that by Theorem \ref{maintheorem}(d)(iii) the image of $\B_3$ is
infinite unless $\ell=10$ or $20$.
\end{example}

\begin{example} Now let $\g=\mathfrak{f}_4$, and $V$ be the simple object in $\CC(\mathfrak{f}_4,q,\ell)$ analogous to the vector representation of $\mathfrak{f}_4$.  If $22\leq\ell$ and even, $\dim\Hom(V,V^{\ot 3})=5$ and the eigenvalues of the image of $\sigma_1$ are $\{q^{-24},q^{-12},q^2,-1,-q^{-6}\}$.  Thus Theorem \ref{maintheorem}(d) implies that the image of $\B_3$ is infinite for $22\leq\ell$ since the projective order of the image of $\sigma_1$ is $\ell$ in these cases.  Notice that in the case $\ell=24$ we have repeated eigenvalues.  Next assume that $\ell$ is odd.  We have $\dim\Hom(V,V^{\ot 3})=5$ when $15\leq\ell$, and the projective order of the image of $\sigma_1$ is $2\ell$ so the image of $\B_3$ is again always infinite by Theorem \ref{maintheorem}(d).
\end{example}

\begin{example}
Consider $\g=\mathfrak{so}_7$ and let $V$ be the simple object in $\CC(\mathfrak{so}_7,q,\ell)$  corresponding to the fundamental spin representation of $\mathfrak{so}_7$.
When $\ell$ is even and $14\leq\ell$ we have $\dim\Hom(V,V^{\ot 3})=4$ and the eigenvalues of the image of $\sigma_1$ are $\chi\{1,q^{12},-q^6,-q^{10}\}$.  Since $q=e^{\pi\im/\ell}$ with $\ell$ even, $po(A)=\ell/2\geq 7$.  So provided $\ell/2\not\in\{7,8,9,10,12,15,15,20,24\}$, $G$ is infinite.  There are two representations with these eigenvalues corresponding to the two choices of
$D=\pm \sqrt{\frac{\la_2\la_3}{\la_1\la_4}}=\pm q^4$.  Choosing $D=q^4$, we obtain the following matrices from \cite{TbWz}:
$$A=\begin {pmatrix} 1& \left( {q}^{8}+{q}^{4}+1 \right) {q}^{4}&-{\frac {{q}^{8}+{q}^{4}+1}{{q}^{2}}}&-{q}^{10}\\0
&{q}^{12}&- \left( {q}^{4}+1 \right) {q}^{2}&-{q}^{10}
\\0&0&-{q}^{6}&-{q}^{10}\\0&0&0&-{
q}^{10}\end {pmatrix}$$
$$B=\begin {pmatrix} -{q}^{10}&0&0&0\\{q}^{6}&-{q}^{6}&0&0\\{q}^{16}&- \left( {q}^{4}+1 \right)
{q}^{12}&{q}^{12}&0\\-{q}^{12}&{q}^{12}+{q}^{8}+{q}^
{4}&-{q}^{8}-{q}^{4}-1&1\end {pmatrix}.$$

We first consider the case $\ell=14$.  In this case we have $A^7=B^7=I$.  We compute that (projectively) $(AB^{-1})^4=I$.  This implies that $G$ is indeed finite!  We compute that $|G/Z(G)|=168$ so that, projectively, $G$ is $\PSL(2,7)$.  We further note that there is an object $U$ so that $\dim\Hom(U,V^{\ot 3})=3$ and the corresponding braid group representation is
irreducible, given in \cite{West}.  The eigenvalues are (up to scaling and Galois conjugation) $\{1,e^{2\pi \im/7},e^{10\pi \im/7}\}$ so that by Theorem 
\ref{maintheorem}(d)(ii), this reprensentation also has finite image.

Next consider the case $\ell=18$.  Under this substitution
the eigenvalues are of the form $\{1,\omega,\omega^2,\alpha\}$ where $\omega$ is a primitive $3$rd root of unity.  So for either choice of $D$ we find that $G$ is a finite imprimitive group by Theorem \ref{maintheorem}(c)(i) and the TW classification.
\end{example}

\begin{example}
Consider $\g=\mathfrak{so}_9$ and let $V$ be the simple object in $\CC(\mathfrak{so}_9,q,\ell)$  corresponding to the fundamental spin representation of $\mathfrak{so}_9$. When $\ell$ is even and $18\leq\ell$ we have $\dim\Hom(V,V^{\ot 3})=5$ and the eigenvalues of the image of $\sigma_1$ are
$\chi\{1,q^8,-q^{14},-q^{18},q^{20}\}$, and $\gamma=\det(A)^{1/5}=\zeta_5q^{12}$ where $\zeta_5$ is a primitive $5$th root of unity.  For simplicity we will assume $\gamma=q^{12}$, although this assumption certainly can affect $|G|$.

For $\ell=18$ the matrix $A$ has repeated eigenvalues, (namely $1$ and $-q^{18}=1$) so Lemma 3.2 of \cite{TbWz}
fails and we cannot conclude that the representation is irreducible in this case.
In fact, the image for $\ell=18$ satisfy relations $A^9=B^9=(A^4(ABA)A^5(ABA))^2$
which, together with the braid relation, imples that the projective image $G$ is a quotient of a group of order $324$.  This implies that the representation is reducible, since any such group cannot have an irreducible representation of dimension $5$.

For $\ell\geq 20$ one does have distinct eigenvalues, so that the corresponding representation is irreducible.  We see that $po(A)=\ell/2$ in these cases, so
provided $\ell/2\not\in\{10,11,12\}$ $G$ is infinite.  The matrices, $A$ and $B$ we obtain are:
$$\begin {pmatrix} 1&{q}^{8}-{q}^{6}+{q}^{4}-{q}^{2}&-{q}^{14}+{q}^{12}-2\,{q}^{10}+{q}^{8}-{q}^{6}&-{q}^{16}+{q}^{14}-{q}^{12}+{
q}^{10}&{q}^{16}\\0&{q}^{8}&-{q}^{14}+{q}^{12}-{q}^{
10}&-{q}^{16}+{q}^{14}-{q}^{12}&{q}^{16}\\0&0&-{q}^{
14}&-{q}^{16}+{q}^{14}&{q}^{16}\\0&0&0&-{q}^{18}&{q}
^{18}\\0&0&0&0&{q}^{20}\end {pmatrix}$$
$$\begin {pmatrix} {q}^{20}&0&0&0&0\\{q}^{18}&-{q}^{18}&0&0&0\\{q}^{16}&-{q}^{16}+{q}^{14}&-{
q}^{14}&0&0\\{q}^{16}&-{q}^{16}+{q}^{14}-{q}^{12}&-{
q}^{14}+{q}^{12}-{q}^{10}&{q}^{8}&0\\{q}^{16}&-{q}^{
16}+{q}^{14}-{q}^{12}+{q}^{10}&-{q}^{14}+{q}^{12}-2\,{q}^{10}+{q}^{8}-
{q}^{6}&{q}^{8}-{q}^{6}+{q}^{4}-{q}^{2}&1\end {pmatrix}.$$
The case $\ell=22$ is interesting: If we set $S=A$ and $T=ABA$, we find that the $\PSL(2,11)$ relations $S^{11}=T^2=(S^4TS^6T)^2=I$ hold (projectively), so that the projective image $G$ is $\PSL(2,11)$.  This is not too surprising since $\PSL(2,\Z)$ is a quotient of $\B_3$.  The cases $\ell=20$ and $\ell=24$ do not yield finite groups.
\end{example}

\end{document}